\documentclass[10pt,english,oneside,a4paper]{amsart}

\usepackage[english]{babel}
\usepackage[utf8]{inputenc}
\usepackage[T1]{fontenc}
\usepackage[top=2.5cm, bottom=2.75cm, left=2cm, right=2cm, a4paper]{geometry}
\usepackage[hidelinks]{hyperref}
\usepackage{amsmath,amsfonts,amssymb,mathtools,amsthm}
\usepackage{listings,verbatim,fancyvrb,xspace,enumerate}
\usepackage{xcolor,graphicx,ytableau,tikz}
\usepackage[all, cmtip]{xy}
\usepackage{fancyhdr}
\usepackage[toc,page]{appendix}
\usepackage{pdfpages}
\usepackage{aligned-overset}

\newtheorem{theorem}{Theorem}[subsection]
\newtheorem*{theorem*}{Theorem}

\newtheorem{corollary}[theorem]{Corollary}
\newtheorem*{corollary*}{Corollary}
\newtheorem{lemma}[theorem]{Lemma}
\newtheorem{proposition}[theorem]{Proposition}
\theoremstyle{definition}
\newtheorem{example}[theorem]{Example}
\newtheorem*{example*}{Example}
\newtheorem*{examples*}{Examples}
\newtheorem{remark}[theorem]{Remark}
\newtheorem{definition}[theorem]{Definition}
\newtheorem{definition-proposition}[theorem]{Definition-Proposition}
\newtheorem{example-definition}[theorem]{Example-Definition}

\newsavebox{\eqbox}
\newenvironment{longequation*} {\begin{lrbox}{\eqbox}$} {$\end{lrbox}\begin{equation*}\resizebox{\linewidth}{!}{\ensuremath{\displaystyle\usebox{\eqbox}}}\end{equation*}}

\makeatletter
\newcommand*\bigcdot{\mathpalette\bigcdot@{.5}}
\newcommand*\bigcdot@[2]{\mathbin{\vcenter{\hbox{\scalebox{#2}{\(\m@th#1\bullet\)}}}}}

\makeatother

\def\bydef{\coloneqq}

\DeclarePairedDelimiter{\Set}{\lbrace}{\rbrace}

\DeclareMathOperator{\Vect}{Vect}

\DeclareMathOperator{\Id}{Id}

\DeclareMathOperator{\Pic}{Pic}

\DeclareMathOperator{\Proj}{Proj}

\DeclareMathOperator{\Ker}{Ker}

\DeclareMathOperator{\GL}{GL}
\DeclareMathOperator{\Diff}{diff}

\DeclareMathOperator{\Wronsk}{Wronsk}

\DeclareMathOperator{\Spec}{Spec}

\DeclareMathOperator{\End}{End}

\DeclareMathOperator{\ST}{ST}

\DeclareMathOperator{\subs}{subs}
\DeclareMathOperator{\coeff}{coeff}

\newcommand{\C}{\mathbb{C}}
\renewcommand{\O}{\mathcal{O}} 
\renewcommand{\P}{\mathbf{P}} 

\newcommand{\N}{\mathbb{N}}
\newcommand{\Z}{\mathbb{Z}}

\let\mathbi\boldsymbol

\author{Antoine Etesse}
\email{antoine.etesse@ens-lyon.fr}
\address{Unité de Mathématiques Pures et Appliquées (UMPA), ENS de Lyon}
\title{On the algebraic structure of differentially homogeneous polynomials.}

\begin{document}
\sloppy

\begin{abstract}
The paper describes the algebraic structure of the graded algebra of differentially homogeneous polynomials in \((N+1)\) variables, of \textsl{order} at most \(k \in \N\). We show that it is a finitely generated algebra, and we exhibit a minimal set of generators. Along the way, we provide a simpler proof of the so-called Schmidt--Kolchin conjecture, proved in \cite{ESK}. 

From the algebraic point of view, this provides a natural compactification of the \(k\)th jet bundle \(J_{k}\P^{N}\) of the \(N\)th-dimensional projective space \(\P^{N}\). From the invariant theoretic point of view, this provides a new example, not covered by known conjectures in the subject, of a unipotent sub-group of \(\GL_{k+1}(\C)\) whose algebra of invariants is finitely generated (and more precisely gives a \textsl{First Fundamental Theorem} for such a group, following the terminology in Invariant Theory)
\end{abstract}
\maketitle

\section*{Introduction}
The goal of this paper is to further the study of \textsl{differentially homogeneous polynomials}, and in particular give a description of its algebraic structure. Just as the graded (\(\C\)-)algebra \(V^{(0)}:=\C[X_{0}, \dotsc, X_{N}]\) of homogeneous polynomials in \((N+1)\) variables, \(N \in \N\), forms the natural \textsl{generalized\footnote{Namely, functions taking values in a line bundle. Here, there is no other choice than \((\O_{\P^{N}}(d))_{d \in \Z}\), as \(\Pic(\P^{N})=\Z \cdot \O_{\P^{N}}(1)\).}} functions on the projective space \(\P^{N}\), the graded algebra
\((V^{(k)})^{\Diff}\) of \textsl{differentially homogeneous polynomials in \((N+1)\) variables of order at most \(k \in \N\)} forms the natural \textsl{generalized}\footnote{Note that \(\Pic(J_{k}\P^{N}) \simeq \Z\).} functions on the \(k\)th jet bundle \(J_{k}\P^{N}\): see Section \ref{sse: preli1} for the formal definition of \((V^{(k)})^{\Diff}\), and see \cite{ESK}[Section 3] for details on the above interpretation of \((V^{(k)})^{\Diff}\).
There are natural inclusions
\[
V^{(0)}=(V^{(0)})^{\Diff} \subset (V^{(1)})^{\Diff} \subset \dotsb,
\]
and one simply denotes \(V^{\Diff}:= (V^{(\infty)})^{\Diff}\) the \textsl{graded algebra of differentially homogeneous (in \((N+1)\) variables)}. The upper index in parenthesis \((\cdot)\) is referred to as the \textsl{order (of derivation)}.

The dimension of the graded components of \(V^{\Diff}\) were conjectured to be equal to \((N+1)^{d}\): this was referred to as the \textsl{Schmidt--Kolchin conjecture} in the literature: see \cite{Schmidt_1979}, \cite{Kolchin}, \cite{Rein-Sit} and \cite{Buium}. We solved this conjecture in \cite{ESK}, and we will provide an alternative, simpler proof in Section \ref{ssse: alternative}: the proof hereby relies on the ideas from \textsl{loc.cit}, but allows to bypass some of the technicalities.
Note that there was a recent preprint \cite{Merker}, in which another completely different proof is provided, via the sole point of view of \textsl{Green--Griffiths vector bundles}. 

In this paper, we describe the algebraic structure of \((V^{(k)})^{\Diff}\), and show the following (see also Theorem \ref{thm: final}):
\begin{theorem}[Main Theorem]
\label{thm: main intro}
The graded algebra \((V^{\Diff})^{(k)}\) is finitely generated by \((N+1)\) generators in degree \(1\) and order \(0\) and, for \(2 \leq i \leq k+1\),
\begin{center}
\(\frac{N(N+1)}{2}\times N^{i-2}\) generators in degree \(i\) and order \((i-1)\).
\end{center}
Furthermore, the set of generators is explicit, and minimal.
\end{theorem}
As a corollary, we obtain an explicit compactification of the \(k\)th jet bundle \(J_{k} \P^{N}\) (see also Definition--Proposition \ref{def-prop: compactification}):
\begin{corollary}
The projective variety
\[
(\P^{N})^{(k)}
:=
\Proj (V^{\Diff})^{(k)}
\]
compactifies in a natural fashion the \(k\)th jet bundle \(J_{k}\P^{N}\).
\end{corollary}
Using the point of view of Invariant Theory (see Section \ref{sse: reformulation}), the previous Theorem \ref{thm: main intro} can be reformulated as follows:
\begin{theorem}[Reformulation of the Main Theorem]
Let \(F\) be \(\C\)-vector space of dimension \((k+1)\), equipped with a nilpotent endomorphism \(J\) of maximal index (i.e. \(J^{k} \neq 0\)). Such an endomorphism induces a linear action of \(\C[J]^{\times} \simeq (\C[T]/(T^{k+1}))^{\times}\) on \(F\).

Under this action, the algebra of quasi-invariant\footnote{These are the elements \(v\) such that there exists \(d \in \N\) such that, for any \(P \in \C[T]\), the following holds:
\[
P(J) \cdot v
=
P(0)^{d} v.
\]
}
polynomials in \((N+1)\) copies of \(F\) is finitely generated by explicit generators
\[
\mathcal{G}_{1}, \dotsc, \mathcal{G}_{k+1}.
\]
Furthermore, this is a minimal set of generators, and they sastisfy:
\begin{itemize}
\item{}
\(|\mathcal{G}_{1}|=(N+1)\);
\item{}
\(|\mathcal{G}_{i}|=\frac{N(N+1)}{2} (N+1)^{i-2}\) for every \(2 \leq i \leq k+1\).
\end{itemize}

\end{theorem}

Such a statement is often referred to as a \textsl{First Fundamental Theorem} in Invariant Theory. As the terminology indicates, there is a \textsl{Second Fundamental Theorem}. Its content deals with the algebraic relations between the generators. We are currently working on it.

Note that the subgroup \(\C[J]^{\times} \subset \GL_{k+1}(\C)\) is an extension by \(\C^{\times}\) of a unipotent subgroup \(U\) of \(\GL_{k+1}(\C)\) (see Section \ref{sse: reformulation}).
In this setting, there is no general theory ensuring that the algebra of (quasi-)invariants is finitely generated, and it is well-known that it cannot be the case in full generality: see e.g. \cite{Popov}. A particular instance of the so-called \textsl{Popov--Pommering conjecture} asserts that finite generation should hold for subgroups \(G \leq \GL_{k+1}(\C)\) which
\begin{itemize}
\item{} are of the above form, i.e. writes
\[
G \simeq U \rtimes \C^{*}
\]
for some unipotent subgroup \(U \leq \GL_{k+1}(\C)\);
\item{} 
are normalized by the subgroup of diagonal matrixes (i.e. by the maximal torus of \(\GL_{k+1}(\C))\).
\end{itemize}
This conjecture is still open, but some particular cases were proved: we refer to \cite{BercziPopov} for an account on what is known, and the recent contributions made in \textsl{loc.cit}.
Observe that, besides the case where \(k=1\) (where \(\C[J]^{\times}\) coincides with the subgroup of upper triangular matrixes), the subgroup \(\C[J]^{\times}\) is very far from being normalized by diagonal matrixes. Accordingly, our result does not fit into the predictions of the Popov--Pommering conjecture. It should rather be interpreted as a transversal result.

Let us end this introduction by giving the rough idea of the proof of the Main Theorem. Keeping the notations introduced above, one wishes to understand
\[
\big(S^{\bigcdot}(F^{\oplus(N+1)})\big)^{U}.
\]
By an elementary observation (following from Lemma \ref{lemma: simple}), one sees that, in order to prove the finite generation property, it is enough to show that, for any \(d \in \N_{\geq 1}\), the space of invariant tensors
\[
(F^{\otimes d})^{U}
\]
admits a \textsl{simple} basis, in the sense that each element in the basis writes as a product, with a control on the growth of the degree of each term (see the proof in Section \ref{sse: fg} for details).
In order to achieve this, we interpret \((F^{\otimes d})^{\otimes}\) as a space of \textsl{harmonic polynomials}, and use their structure, well studied in the literature (see Section \ref{sse: harmonic} and \ref{sse: DCP}).
The description of the explicit generators is a bit technical, but completely manageable in our situation (see Section \ref{sse: gen}).

\section{Differentially homogeneous polynomials}~

\label{se: preli}

For this whole section, we fix \(N \in \N_{\geq 1}\) a positive natural number.

\subsection{Definitions and first properties}~
\label{sse: preli1}

\subsubsection{Definitions}
\label{ssse: def}

Recall the usual definition of \textsl{differential polynomials in \((N+1)\) variables}:
\begin{definition}[Differential polynomials]~
Consider, for any \(k \in \N\), the formal variables
\[
 X^{(k)}\bydef (X_{0}^{(k)}, \dotsc, X_{N}^{(k)}).
 \]
 The complex algebra of \textsl{differential polynomials in \((N+1)\) variables} is defined as
\[
V\bydef \C\big[(X^{(k)})_{k \in \N}].
\]
An element in this algebra is thus called a \textsl{differential polynomial} in the variables \(X=(X_{0}, \dotsc, X_{N})\).
\end{definition}

Note that the \(\C\)-algebra \(V\) is naturally graded with the usual grading for polynomials (where each variable has degree \(1\))
\[
V=\bigoplus_{d \in \N}
V_{d},
\]
where \(V_{d}\) is the \(d\)th graded component. Furthermore, the algebra of differential polynomials \(V\) is also naturally filtered by the maximal \textsl{order of derivation}
\[
(0) \subset V^{(0)} \subset V^{(1)} \subset \dotsb \subset V^{(k)} \subset \dotsb
\]
where, for any \(k \in \N\), one sets \(V^{(k)}\bydef \C\big[(X^{(i)})_{0 \leq i \leq k}]\).

There is a natural linear action of the group of invertible formal series in one variables \(\C[[T]]^{\times}\) on \(V\) defined as follows. 
For any differential polynomial \(P \in V\) and any invertible formal serie \(\alpha \in \C[[T]]^{\times}\), form a new differential polynomial \(\alpha \cdot P \in V\) by setting
\[
\alpha\cdot P
\bydef 
P\big((\alpha X)^{(0)}, (\alpha X)^{(1)}, \dotsc\big)_{\vert T=0}.
\]
Here, for any \(k \in \N\), the symbol \((\alpha X)^{(k)}\) is, by definition,
\[
(\alpha X)^{(k)}
\bydef
\sum\limits_{i=0}^{k}
\binom{k}{i}
\alpha^{(k-i)}X^{(i)}.
\]
Namely, one simply applies formally the usual Leibnitz rule. The fact that it does indeed define an action follows from a formal algebraic identity.
And one immediately sees that the action is indeed linear.

This linear action allows in turn to define the graded (\(\C\)-)sub-algebra of \textsl{differentially homogeneous polynomials}:
\begin{definition}[Differentially homogeneous polynomials]~
A differential polynomial \(P \in V\) is called \textsl{differentially homogeneous of degree \(d\)} if and only if, for any \(\alpha \in \C[[T]]^{\times}\), the following equality holds:
\[
\alpha \cdot P
=
\alpha(0)^{d}
P.
\]
\end{definition}
Said otherwise, differentially homogeneous polynomials are the \textsl{quasi-invariant} elements of \(V\) under the action of \(\C[[T]]^{\times}\).
In particular, note that a differentially homogeneous polynomial of degree \(d\) is necessarily homogeneous of degree \(d\). Denote accordingly by 
\[
V^{\Diff}_{d} \subset V_{d}
\]
the sub-vector space of differentially homogeneous polynomials of degree \(d\), and set
\[
V^{\Diff}
:=
\bigoplus 
V_{d}^{\Diff}
\subset V
\]
the graded sub-algebra of \textsl{differentially homogeneous polynomials}.
Note that the filtration descends to the sub-algebra \(V^{\Diff}\), and we will denote by
\[
(V^{\Diff})^{(k)}=(V^{(k)})^{\Diff}
\]
the sub-vector space of differentially homogeneous polynomials of order at most \(k\).

\subsubsection{Factoriality}~
\label{ssse: factoriality}
The Lemma \ref{lemma: product diff} below allows to show that the the sub-algebra \(V^{\Diff}\) inherits the factoriality of \(V\):
\begin{lemma}
\label{lemma: product diff}
Let \(P, Q \in V\) be non-zero differential polynomials such that their product
is differentially homogeneous. Then both \(P\) and \(Q\) are differentially homogeneous.
\end{lemma}
\begin{proof} 
Let \(\alpha \in \C[[T]]^{\times}\), and consider its coefficients \(\lambda^{(k)}:=\frac{\alpha^{(k)}(0)}{k!}\) as formal variables.
It is clear that one has the equality
\[
\alpha \cdot (PQ)
=
(\alpha \cdot P) (\alpha \cdot Q),
\]
so that either both \(P\) and \(Q\) are differentially homogeneous, or both are not. Arguing by contradiction, suppose that the latter case holds. Therefore, up to exchanging the roles of \(P\) and \(Q\), there exists \(k \geq 1\) such that
\begin{itemize}
\item{} \(\alpha \cdot P\) can be considered as a \textsl{non-constant} polynomial in \(A[\lambda^{(k)}]\), where 
\[
A=V[\lambda, \lambda^{(1)}, \dotsc, \lambda^{(k-1)}];
\]
\item{} \(\alpha \cdot Q\) can be considered as a polynomial in \(A[\lambda^{(k)}]\) (possibly constant, i.e. in \(A\)).
\end{itemize}
Since \(V\) is integral, so is \(A\), and one deduces accordingly that
\[
(\alpha \cdot P) (\alpha \cdot Q)
\]
is a non-constant polynomial in \(A[\lambda^{(k)}]\). This contradicts the fact that this must be a differentially homogeneous polynomial, and thus finishes the proof of the lemma.
\end{proof}

\begin{corollary}
\label{cor: factoriality}
The graded algebra \(V^{\Diff}\) of differentially homogeneous polynomials is factorial.
\end{corollary}
\begin{proof}
The existence of a factorization by irreducible elements for any differentially homogeneous polynomials follows immediately from the previous Lemma \ref{lemma: product diff} and the fact that \(V\) is itself factorial. The unicity up to permutation of the factors follows similarly.
\end{proof}

\subsection{Reformulation via the formalism of Invariant Theory}~
\label{sse: reformulation}
Let us fix \(k \in \N_{\geq 1}\) a positive order of truncation. 
As the title indicates, the goal of this section is to reformulate the definition of \((V^{(k)})^{\Diff}\) via the formalism of invariant theory.

Let \(F_{k}\) be a \(\C\)-vector space of dimension \((k+1)\), equipped with a nilpotent endomorphism \(J_{k}\) of maximal index (i.e. \(J_{k}^{k} \neq 0\)). Such an endomorphism induces a natural linear action of 
\[
\C[J_{k}]^{\times} \simeq \big(\C[T]/(T^{k+1})\big)^{\times}
\]
on the vector space \(F_{k}\). Note that such a group is highly \textsl{non-reductive}, as it is the extension of the unipotent group
\[
U_{k}
:=
\{ P \in \C[J_{k}]^{\times}
\ 
|
\
P(0)=1
\}
\]
by the multiplicative group \(\C^{\times}\). In such a setting, the main object of what is usually called \textsl{Invariant Theory} is the study of the elements in the symmetric
algebra of several copies of \(F\) (say, \((N+1)\) for the sake of this section) that are invariant under the natural (diagonal) action of \(U_{k}\), namely:
\[
\Big(
S^{\bigcdot}
\big(
\underbrace{F_{k}\oplus \dotsb \oplus F_{k}}_{\times (N+1)}
\big)
\Big)^{U_{k}}.
\]
\begin{remark}
Note that two nilpotent endomorphisms of maximal index are always conjugate, hence the algebraic structure of the invariant algebra does not depend on the choice of \(J_{k}\): this is why we do not specify the dependency in \(J_{k}\) in the definition of \(U_{k}\).
\end{remark}

With this point of view, one recovers the \(\C\)-algebra \((V^{(k)})^{\Diff}\):
\begin{lemma}
There is a natural isomorphism of \(\C\)-algebra
\[
(V^{(k)})^{\Diff}
\simeq
\Big(
S^{\bigcdot}
\big(
\underbrace{F\oplus \dotsb \oplus F_{k}}_{\times (N+1)}
\big)
\Big)^{U_{k}}.
\]
\end{lemma}
\begin{proof}
Consider \(F_{k}\) as \(\C^{k+1}\) endowed with its canonical basis \((f_{i})_{0 \leq i \leq k}\), and pick the endomorphism \(J_{k}\) defined as follows on the canonical basis
\[
J_{k}(f_{i})=i\times f_{i-1},
\]
with the convention that \(J_{k}(f_{0})=0\).

Denote by the symbols \((X_{i}^{(j)})_{0 \leq j \leq k}\), \(0 \leq i \leq N\), the canonical basis of the \((i+1)\)th copy of \(F_{k}\). Hence, the symmetric algebra
\[
S^{\bigcdot}
\big(
\underbrace{F_{k}\oplus \dotsb \oplus F_{k}}_{\times (N+1)}
\big)
\]
identifies with \(V^{(k)}\). To conclude, it suffices to observe that for any \(\alpha \in \C[[T]]^{\times}\), one has the following equality
\[
\alpha(T) \cdot P
=
\alpha(J) \cdot P
\]
where, on the left-hand side, \(P\) is seen as an element of \(V^{(k)}\) and the action is the one described in Section \ref{sse: preli1}.
\end{proof}

\subsection{An alternative proof of the Schmidt--Kolchin conjecture}~
\label{sse: alternative}
The goal of this section is to provide an alternative proof of the Schmidt--Kolchin conjecture, solved in \cite{ESK}. Recall that this conjecture says that the following equality holds for any \(d \in \N\):
\[
\dim (V_{d}^{\Diff})
=
(N+1)^{d}.
\]
The method in \textsl{loc.cit} to prove the above equality is based on the study of invariant \textsl{highest weight vectors} (with respect to the natural linear action of \(\GL_{N+1}(\C)\) on \(V_{d}^{\Diff})\) in \(V_{d}^{\Diff}\). The approach taken here, and detailed below is, in a sense, transversal and more direct.
Let us note that this alternative proof allows to by-pass some technicalities, but relies on what was done in \cite{ESK}.

\subsubsection{A simple, but important, Lemma}~
\label{ssse: lemma inv}

Let \(E\) be \(\C\)-vector space of finite dimension, and \(G\) a group acting linearly on \(E\) (i.e. \(E\) is a representation of \(G\)). The simple but important observation is the following:
\begin{lemma}
\label{lemma: simple}
For any \(d \geq 1\), there is a natural surjective map of \(GL(E)\)-representation
\[
p\colon
(E^{\otimes d})^{G}
\twoheadrightarrow
(S^{d}E)^{G}.
\]
\end{lemma}
\begin{proof}
The natural surjective map of \(\GL(E)\)-representation is the restriction of the following surjective map of \(GL(E)\)-representation:
\[
p\colon
\left(
\begin{array}{ccc}
E^{\otimes d} & \longrightarrow  &  S^{d}E
\\
e_{1} \otimes \dotsb \otimes e_{n} & \longmapsto  & e_{1} \dotsb e_{n}
\end{array}
\right).
\]
This map admits a right-inverse injective map of \(\GL(E)\)-representation
\[
i\colon
\left(
\begin{array}{ccc}
S^{d}E & \longrightarrow  &  E^{\otimes d} 
\\
e_{1}\dotsb e_{n} & \longmapsto  & \frac{1}{d!} \sum\limits_{\sigma \in \mathcal{S}_{d}} e_{\sigma(1)}\otimes \dotsb \otimes e_{\sigma(n)}
\end{array}
\right),
\]
namely, \(p\circ i = \Id_{S^{d}E}\). 
Since \(S^{d}E \simeq i(S^{d}(E))\) is a \(\GL(E)\)-subrepresentation of \(E^{\otimes d}\), and since \(\GL(E)\) is reductive, there is a direct sum decomposition of \(\GL(E)\)-representations
\[
E^{\otimes d}
=
i(S^{d}E) \oplus Q.
\]

Now, since the action of \(G\) is linear, it preserves the decomposition of \(E^{\otimes d}\) into \(\GL(E)\)-subrepresentation. Furthermore, it commutes with the maps \(p\) and \(i\).
In particular, one has
\[
(E^{\otimes d})^{G}
=
i(S^{d}E)^{G} \oplus Q^{G}
=i(S^{d}(E)^{G}) \oplus Q^{G}.
\]
The surjectivity of \(p_{\vert (E^{\otimes d})^{G}}\colon (E^{\otimes d})^{G} \to (S^{d}E)^{G}\) follows immediately.
\end{proof}

Such a statement is nothing less than a simple observation, but we believe that it is an important one. 
As the article will show, the full space of invariants \((E^{\otimes d})^{G}\) is sometimes easier to understand than the subspace \((S^{d}E)^{G}\). Informally speaking, the 
variables in \(E^{\otimes d}\) are completely separated, making the search for invariants more transparent and  easier than in \(S^{d}E\), where the variables are somewhat \textsl{entangled}.

\subsubsection{The alternative proof}~
\label{ssse: alternative}
One wishes to understand the vector space of \(U_{k}\)-invariants in
\[
S^{d}
\big(
\underbrace{F_{k}\oplus \dotsb \oplus F_{k}}_{\times (N+1)}
\big)
\simeq 
V_{d}^{(k)},
\]
for \(k\) large enough. Recall that there are natural inclusions 
\[
V^{(0)} \subset V^{(1)} \subset \dotsb \subset V^{(k)} \subset \dotsb.
\] 
We will show that set of invariants stagnates at \(V_{d}^{(d-1)}\), and exhibit all of them.

First, let us construct a natural family of invariant in \(V_{d}^{(d-1)}\). The construction is the complete analogue of the one described in \cite{ESK}[Section 2.1], but displayed with the formalism detailed in Section \ref{sse: reformulation}. 
\begin{definition-proposition}[Construction of a natural family of invariants]~
\label{def-prop: inv rep} 
Let \(\mathbi{\alpha}:=(\mathbi{\alpha_{0}}, \dotsc, \mathbi{\alpha_{N}})\) be a sequence of uples such that 
\begin{itemize}
\item{} for any \(0 \leq i \leq N\),
\begin{itemize}
\item{} either \(\mathbi{\alpha_{i}}\) is empty;
\item{} either \(\mathbi{\alpha_{i}}\) is a sequence of natural numbers, i.e. 
\[
\mathbi{\alpha_{i}}
=
(\alpha_{i,1}, \dotsc, \alpha_{i,r_{i}}),
\]
with \(r_{i} \geq 1\);
\end{itemize}
\item{} one has the equality:
\[
r_{0} + \dotsc + r_{N}
=
d,
\]
with the convention that \(r_{i}=0\) if \(\mathbi{\alpha_{i}}\) is empty.
\end{itemize}
Denote by \((X_{i}^{(j)})_{0 \leq j \leq (d-1)}\), \(0 \leq i \leq N\), the canonical basis of the \((i+1)\)th copy of \(F_{d-1}\) in \(F_{d-1}^{\oplus (N+1)}\), and set
\[
v_{0}:=(X_{0}^{(0)}, \dotsc, X_{0}^{(d-1)})^{\intercal}, \dotsc, v_{N}:=(X_{N}^{(0)}, \dotsc, X_{N}^{(d-1)})^{\intercal}.
\] 
Then, the polynomial
\[
W_{\mathbi{\alpha}}
:=
\det(J_{d-1}^{\alpha_{0,1}} \cdot v_{0}, \dotsc, J_{d-1}^{\alpha_{0,r_{0}}} \cdot v_{0}, \dotsc, J_{d-1}^{\alpha_{\alpha_{N,1}}} \cdot v_{N}, \dotsc, J_{d-1}^{\alpha_{N,r_{N}}} \cdot v_{N})
\]
is a \(U_{d-1}\)-invariant polynomial of degree \(d\) (Here, the group \(\GL(F_{d-1})\) acts component-wise on the \(d\)-uples \(v_{0}, \dotsc, v_{N}\)).
\end{definition-proposition}
\begin{proof}
Observe that the diagonal action of \(J_{d-1}\) on the coordinates of \(v_{i}\) writes
\[
J_{d-1} \cdot (X_{i}^{(0)}, \dotsc, X_{i}^{(d-1)})^{\intercal}
=
J_{d-1}^{\intercal}(X_{i}^{(0)}, \dotsc, X_{i}^{(d-1)})^{\intercal},
\]
where the left hand side of the equality is a matrix product, with the identification of \(J_{d-1}\) with its representative matrix in the canonical basis. The invariance now follows immediately from the multiplicative property of the determinant: for any \(\alpha \in U_{d-1} \subset \C[J_{d-1}]^{\times}\), one has the equalities
\begin{eqnarray*}
\alpha \cdot W_{\mathbi{\alpha}}
&
=
&
\det
\big(
(J_{d-1}^{\alpha_{\alpha_{0,1}}}\alpha) \cdot v_{0}, \dotsc, (J_{d-1}^{\alpha_{N,r_{N}}}\alpha)\cdot v_{N}
\big)
\\
&
=
&
\det
\big(
\alpha^{\intercal}(J_{d-1}^{\alpha_{\alpha_{0,1}}} \cdot v_{0}), \dotsc, \alpha^{\intercal}(J_{d-1}^{\alpha_{N,r_{N}}} \cdot v_{N})
\big)
\\
&
=
&
\det(\alpha^{\intercal})
W_{\mathbi{\alpha}}
\\
&
=
&
W_{\mathbi{\alpha}}.
\end{eqnarray*}
\end{proof}
From this family of invariants, one can
\begin{itemize}
\item{} 
extract a free family made of \((N+1)^{d}\) invariants: this is \cite{ESK}[Prop 2.1.2];
\item{}
show that the previous free family actually spans the whole family: this follows from \cite{ESK}[Appendix A]
\end{itemize}
\begin{remark}
The second part of this statement was one of the technical parts in \cite{ESK}: it was used to show that the span of the above free family of \((N+1)^{d}\) invariants is actually a representation of \(\GL_{N+1}(\C)\) (i.e. is stable under the natural action of \(\GL_{N+1}(\C)\)). We refer to \textsl{loc.cit}[Theorem 2.1.8] for details.
\end{remark}

Second, let us fix \(k \geq (d-1)\), and, following Lemma \ref{lemma: simple}, let us rather study the \(U_{k}\)-invariants in the vector space 
\[
(\underbrace{F_{k}\oplus \dotsb \oplus F_{k}}_{\times (N+1)})^{\otimes d}.
\]
Such a vector space splits as a direct sum, with each summand isomorphic to
\(
F_{k}^{\otimes d}.
\)
Note that the action of \(U_{k}\) is compatible with respect to this decomposition, so that one is reduced to studying the space of invariants
\[
(F_{k}^{\otimes d})^{U}.
\]
It was shown in \cite{ESK}[Lemma 2.1.3]] that, for any \(k \geq (d-1)\), this space has the same dimension \(d!\). In particular, one has the following equality for any \(k \geq (d-1)\)
\[
(F_{d-1}^{\otimes d})^{U_{d-1}}
=
(F_{k}^{\otimes d})^{U_{k}},
\]
where \(F_{d-1}\) is naturally seen as a subspace of \(F_{k}\)\footnote{Note that, in the setting of differentially homogeneous polynomials, the data of the spaces
\((F_{k},J_{k})_{k \in \N}\) is compatible, in the sense that there are natural inclusions \(F_{0} \subset F_{1} \subset \dotsb \subset F_{k} \subset \dotsb\) such that the following holds for any \(k\):
\[
(J_{k})_{\vert F_{k-1}}=J_{k-1}
\
\text{and}
\
J_{k}(F_{k}) \subset F_{k-1}.
\]
}.
Furthermore, an explicit basis can be provided, whose elements have same exact shape as the elements defined in Definition-Proposition \ref{def-prop: inv rep}. 
More precisely, denote by \((Y_{i}^{(j)})_{0 \leq j \leq d-1}\), \(1\leq i \leq d\), the canonical basis of the \(i\)th copy of \(F_{d-1}\) in the tensor product \(F_{d-1}^{\otimes d}\), and set
\[
v_{1}:=(Y_{1}^{(0)}, \dotsc, Y_{d}^{(d-1)})^{\intercal}, \dotsc, v_{d}:=(Y_{1}^{(0)}, \dotsc, Y_{d}^{(d-1)})^{\intercal}.
\] 
Using the same notations as in Definition-Proposition \ref{def-prop: inv rep}, one shows that
\[
(W_{(\alpha_{1}, \dotsc, \alpha_{d})})_{0 \leq \alpha_{i} \leq (i-1)}
\]
forms a basis of \((F_{d-1}^{\otimes d})^{U_{d-1}}\). Note that, by construction, these invariants identify as \textsl{multi-linear} polynomials in the formal variables \((Y_{i}^{(j)})_{1 \leq i \leq d; 0 \leq j \leq (d-1)}\): this identification is convenient, and will often be made in the rest of the paper.

To conclude the proof, it suffices to understand the action of the natural projection map from Lemma \ref{lemma: simple}
\[
p\colon
\big((\underbrace{F_{d-1}\oplus \dotsb \oplus F_{d-1}}_{\times (N+1)})^{\otimes d}\big)^{U_{d-1}}
\twoheadrightarrow
\Big(S^{d}
\big(
\underbrace{F_{d-1}\oplus \dotsb \oplus F_{d-1}}_{\times (N+1)}
\big)
\Big)^{U_{d-1}}
\]
on each summand. One then observes that it simply corresponds to a certain substitution of variables (keeping the notations introduced so far) of the form:
\[
(Y_{1}^{(j)}=X_{i_{1}}^{(j)})_{0 \leq j \leq d-1}, \dotsc, (Y_{d}^{(j)}=X_{i_{d}}^{(j)})_{0 \leq j \leq d-1},
\]
with \(i_{1}, \dotsc, i_{d}\) belonging to set \(\{0, \dotsc, N\}\). In particular, these substitutions applied to the family
\[
(W_{(\alpha_{1}, \dotsc, \alpha_{d})})_{0 \leq \alpha_{i} \leq (i-1)}
\]
yield elements in the family described in Definition-Proposition \ref{def-prop: inv rep}. By the first part of the proof, this shows accordingly that
\[
(V_{d}^{\Diff})^{(d-1)}
\simeq
\Big(S^{d}
\big(
\underbrace{F_{d-1}\oplus \dotsb \oplus F_{d-1}}_{\times (N+1)}
\big)
\Big)^{U_{d-1}}
\]
is of dimension \((N+1)^{d}\). Finally, as a simple consequence of what was seen in the second part of the proof, one has the following equality for any \(k \geq (d-1)\)
\[
(V_{d}^{(k)})^{\Diff}
=
(V_{d}^{(d-1)})^{\Diff}.
\]
Therefore, the proof of the equality
\(
\dim(V_{d}^{\Diff})
=
(N+1)^{d}
\)
is now complete.

\section{Link with harmonic polynomials}~
\label{se: harmonic}
 For this whole section, let us fix \(N \in \N_{\geq 1}\) a positive natural number, and \(k \in \N\) an order of truncation. Keeping the notations introduced in the previous Section \ref{se: preli}, our goal is to understand the algebraic structure of the graded algebra \((V^{\Diff})^{(k)}\). Following the ideas and scheme of proof given in Section \ref{sse: alternative}, one is first lead to study, for any \(d \in \N\), the space of invariant tensors
 \[
(F_{k}^{\otimes d})^{U_{k}}.
 \]
 
In the first two Sections \ref{sse: harmonic} and \ref{sse: DCP}, we explain how this space can be interpreted as a space \textsl{harmonic polynomials}: this was already observed\footnote{It goes without saying that, at that time, we did not know what \textsl{harmonic polynomials} were.} in \cite{ESK} in a special case (allowing to quickly deduce that the dimension of \((F_{d-1}^{\otimes d})^{U_{d-1}}\) is equal to \(d!\)).
In the third Section \ref{sse: structural}, we detail known results concerning such harmonic polynomials, following \cite{BG}.

\subsection{Invariant tensors and harmonic polynomials}~
\label{sse: harmonic}
In order to lighten notations, let us denote \(F:=F_{k}\), \(E:=F^{\otimes d}\), \(J:=J_{k}\) and \(U:=U_{k}\); recall (see Section \ref{sse: reformulation}) that \(J\) is a nilpotent endomorphism of \(F\) with maximal index, i.e. \(J^{k} \neq 0\).
Note that the set of invariant tensors \(E^{U}\) can be described as follows
\[
E^{U}
=
\{
e \in E=F^{\otimes d} 
\
|
\
\forall \alpha \in \C, \ (I_{d} + \alpha J)\cdot e =  e
\}.
\]
This immediately results from the fact that any polynomial \(P \in \C[T], P(0)=1,\) decomposes as a product of monomials of the form \((1+\alpha T)\), \(\alpha \in \C^{\times}\). That being said, the condition for a tensor \(e \in E\) to be invariant can now be conveniently expressed (see also \cite{ESK}[Section 2.4]). To this end, consider for any \(1 \leq \ell \leq k\) the endomorphism \(J^{(\ell)}\in \End(E) \) defined as follows on pure tensors:
\[
J^{(\ell)}(v_{1} \otimes \dotsb \otimes v_{d})
\bydef
\sum\limits_{1 \leq i_{1}\neq \dotsb \neq i_{\ell} \leq d}
(\otimes \prod_{i=1}^{d})(J^{\delta_{i, \Set{i_{1}, \dotsc, i_{\ell}}}}v_{i}),
\]
where, for any set \(I \subset \N\), one sets
\[
\delta_{\cdot,I}\colon \N \to \Set{0,1}
\]
to be the function that is equal to \(1\) if \(i \in I\), and zero otherwise. The key observation is now the following:
\begin{lemma}
\label{lemma: obs}
For any \(v \in E\), and any \(\alpha \in \C\), the following equality holds:
\[
(\Id + \alpha J)
\cdot
v
=
\alpha^{d}J^{(d)}(v) + \alpha^{d-1} J^{(d-1)}(v) + \dotsb + \alpha J^{(1)}(v) + v.
\]
\end{lemma}
\begin{proof}
It suffices to check the equality on pure tensors \(v_{1} \otimes \dotsb \otimes v_{d} \in E\). The equality then follows by expanding by multi-linearity
\[
(v_{1} + \alpha Jv_{1})\otimes \dotsb \otimes (v_{d} + \alpha Jv_{d}).
\]
\end{proof}

Accordingly, one immediately deduces the following description of invariant tensors:
\[
E^{U}
=
\bigcap_{\ell=1}^{d}
\Ker(J^{(\ell)}).
\]
Now, to conveniently study \(E^{U}\), proceed as follows.
Fix \(f \in F\) a vector such that \(J^{k}f \neq 0\), and consider the following basis of \(E\):
\[
\Big(J^{\mathbi{\alpha}}f:=J^{\alpha_{1}}f \otimes \dotsb \otimes J^{\alpha_{d}}f\Big)_{0 \leq \alpha_{1}, \dotsc, \alpha_{d} \leq k}.
\]
Consider the injective morphism of vector spaces
\[
g\colon
\left(
\begin{array}{ccc}
 E & \longrightarrow  & \C[Z_{1}, \dotsc, Z_{d}]  \\
  J^{\mathbi{\alpha}}f & \longmapsto  & \frac{Z_{1}^{k-\alpha_{1}}}{(k-\alpha_{1})!} 
\dotsb
\frac{Z_{d}^{k-\alpha_{d}}}{(k-\alpha_{d})!} 
\end{array}
\right),
\]
and observe that the action of the endomorphism \(J^{(\ell)}\) on \(E\) translates into the action of the differential operator
\(
\sum\limits_{1\leq i_{1} \neq \dotsb \neq i_{\ell} \leq d} \frac{\partial^{\ell}}{\partial X_{i_{1}} \dotsb \partial X_{i_{\ell}}}
\)
on \(\C[Z_{1}, \dotsc, Z_{d}]\). Said otherwise, one has the equality
\[
g \circ J^{(l)}
=
\big(
\sum\limits_{1\leq i_{1} \neq \dotsb \neq i_{\ell} \leq d} \frac{\partial^{\ell}}{\partial X_{i_{1}} \dotsb \partial X_{i_{\ell}}}
\big) \circ g
\]
for any \(1 \leq \ell \leq d\): this is a simple computational check. 

Recall the following notations from differential algebra.
To any polynomial \(Q \in \C[Z_{1}, \dotsc, Z_{d}]\) is associated a differential operator, denoted by \(Q(\partial)\), obtained by doing formally the substitutions:
\[
 Z_{1}^{\alpha_{1}}\dotsc Z_{d}^{\alpha_{d}}
 \mapsto 
 \frac{\partial^{\alpha_{1}+\dotsb +\alpha_{d}}}{\partial Z_{1}^{\alpha_{1}} \dotsb \partial Z_{d}^{\alpha_{d}}}.
 \]
 One thus obtains a linear differential operator: denote by \(Q^{\perp}\) its vector space of polynomials solutions, i.e.
 \[
 Q^{\perp}
 :=
 \{P \in \C[Z_{1}, \dotsc Z_{d}] 
 \
 |
 \
 Q(\partial)(P)=0
 \}.
 \]
More generally, to any ideal \(J \subset \C[Z_{1}, \dotsc, Z_{d}]\), denote by \(J^{\perp}\) the vector space of polynomials satisfying all the differential equations associated to elements in \(J\), namely:
 \[
 J^{\perp}
 :=
\{P \in \C[Z_{1}, \dotsc Z_{d}] 
 \
 |
 \
 \forall Q \in J,
 Q(\partial)(P)=0
 \}.
 \]
 Note that if \(J=(Q_{1}, \dotsc, Q_{p})\), then one trivially has the equality \(J^{\perp}=\bigcap_{i=1}^{p} Q_{i}^{\perp}\).

Now, denote by \(I \subset \C[Z_{1}, \dotsc, Z_{d}]\) the ideal spanned by symmetric polynomials in \(d\) variables, and consider the following ideal
\[
I_{k}
:=
I + (Z_{1}^{k+1}, \dotsc, Z_{d}^{k+1}).
\]
What we have just seen in this Section \ref{sse: harmonic} can be summed up in the following statement:
\begin{proposition}
\label{prop: reformulation}
There is a natural isomorphism 
\[
E^{U}
=
(F_{k}^{\otimes d})^{U_{k}}
\simeq
I_{k}^{\perp}.
\]
\end{proposition}
\begin{proof}
What was seen above allows to identify \(E^{U}\) with the set of polynomials \(P \in \C[Z_{1}, \dotsc, Z_{d}]\) such that the following holds for any \(1 \leq i \leq d\)
\[
\deg_{Z_{i}}(P) \leq k
\
\
\&
\
\
e_{i}(\partial)(P)=0,
\]
where \(e_{i}\) denotes the \(i\)th elementary polynomial in \(d\) variables. The statement now follows from the trivial equivalence
\[
\deg_{Z_{i}}(P)
\leq k
\iff
\frac{\partial^{k+1}}{\partial Z_{i}^{k+1}}(P)=0.
\]
\end{proof}

The (finite dimensional) vector space \(I_{k}^{\perp}\) is a particular instance of what is called in the literature a  \textsl{space of harmonic polynomials}, whose study goes back to DeConcini and Procesi in \cite{DCP}. In the next Section \ref{sse: DCP}, we will recall their definition, and justify that \(I_{k}^{\perp}\) does indeed define a space of harmonic polynomials.

\subsection{DeConcini--Procesi ideals, and \(\mu\)-harmonic polynomials}~
\label{sse: DCP}
Our reference for this part is \cite{BG}. 
We therefore adopt their conventions regarding partitions and Young tableaux. We briefly recall them in the first Section \ref{ssse: not}, along with the main characters of this Section \ref{sse: DCP}, namely the \textsl{DeConcini--Procesi ideals} and the associated \textsl{\(\mu\)-harmonic polynomials} (see Definition \ref{def: harmonic} below).
In the second Section \ref{ssse: DCP}, we justify that our ideal \(I_{k}\) appears as a particular DeConcini--Procesi ideal, so that the invariants \(E^{U}\simeq I_{k}^{\perp}\) can indeed be interpreted as a space of (\(\mu_{k}\)-)harmonic polynomials (for a well-chosen partition \(\mu_{k}\)).

\subsubsection{Notations and definitions}~
\label{ssse: not}
Let \(d \in \N_{\geq 1}\) be a positive integer. A \textsl{partition} \(\mu=(\mu_{1}, \dotsc, \mu_{d})\) of \(d\) is a non-decreasing sequence of \(d\) integers
\[
0 \leq \mu_{1} \leq \dotsb \leq \mu_{d}
\]
whose sum equals to \(d\). To any such partition \(\mu\) is associated a \textsl{Young diagram \(\mathcal{T}\) of shape \(\mu\)}, which is a collection of cells arranged in left-justified rows: the first row contains \(\mu_{1}\) cells, the second \(\mu_{2}\) cells, and so on. By convention, a row with \(0\) cell is an empty row.
By construction, if we read the Young diagram of \(\mathcal{T}\) from top to bottom via the rows, we recover the partition \(\mu\). If instead we read it from right to left via the columns, we find the so-called \textsl{conjugate partition} \(\mu'=(\mu_{1}' \leq \dotsb \leq \mu_{d}')\). Denote, for \(1 \leq k \leq d\)
\[
d_{k}(\mu)
:=
\mu_{1}' + \dotsb + \mu_{k}',
\]
and note that it is equal to the sum of the length of the last \(k\) columns of \(\mathcal{T}\).
Finally, recall that a \textsl{Young tableau with entries in \(\Set{1, \dotsc, d}\)} is a filling \(T\) of the boxes of the Young diagram \(\mathcal{T}\) with the numbers \(\Set{1, \dotsc, d}\). It is called \textsl{injective} if each number appears exactly once. An injective Young tableau with entries in \(\Set{1, \dotsc, d}\) is called \textsl{standard} if 
\begin{itemize}
\item{} browsing from left to right a line of the diagram, the sequence of numbers is increasing;
\item{} browsing from bottom to top a column of the diagram, the sequence of numbers is increasing.
\end{itemize}
The set of standard tableau with shape \(\mu\) is denoted \(\ST(\mu)\).

We can now define the so-called \textsl{DeConcini--Procesi ideals}. Fix \(d \in \N_{\geq 1}\) a natural number, and \(\mu\) a partition of \(d\). For any set
\(
S \subset \mathcal{Z}_{d}:=\{Z_{1}, \dotsc, Z_{d}\}
\)
and any \(r \in \N_{\geq 1}\), denote by \(e_{r}(S)\) the \(r\)th elementary symmetric polynomial in the variables in \(S\). Now, define the \textsl{DeConcini--Procesi ideal associated to \(\mu\)}, denoted \(I_{\mu} \), as the ideal in \(\C[Z_{1}, \dotsc, Z_{d}]\) generated by the partial elementary symmetric polynomials
\[
\mathcal{C}_{\mu}
:=
\{
e_{j}(S)
\
|
\
S \subset \mathcal{Z}_{d},
|S|=i
\
\&
\
i-d_{i}(\mu) < j \leq i
\}.
\]
Finally, let us end this Section \ref{ssse: not} with the definition of \textsl{\(\mu\)-harmonic polynomials}:
\begin{definition}[Harmonic polynomials]
\label{def: harmonic}
A \(\mu\)-harmonic polynomial is  an element in \(I_{\mu}^{\perp}\).
\end{definition}

\subsubsection{The ideal \(I_{k}\) is a DeConcini--Procesi ideal}~
\label{ssse: DCP}
Write the euclidean division of \(d\) by \((k+1)\), \(d=q(k+1)+r\), and consider the partition
\[
\mu_{k}
:=
(\underbrace{q, \dotsc,q}_{\times (k+1-r) }, \underbrace{(q+1), \dotsc, (q+1)}_{\times r}).
\]
The following observation is elementary, but crucial to us:
\begin{proposition}
\label{prop: DCP}
The ideal \(I_{k}\) coincides with the DeConcini--Procesi ideal \(I_{\mu_{k}}\).
\end{proposition}
\begin{proof}
Note that it is clear from the definition of \(I_{\mu_{k}}\) that it contains the ideal \(I\) of symmetric polynomials. 
In order to prove that \((Z_{1}^{k+1}, \dotsc, Z_{d}^{k+1})\) belongs to \(I_{\mu_{k}}\), note that, by symmetry, it is enough to prove that \(Z_{d}^{k+1}\) belongs to \(I_{\mu_{k}}\).
Note also that the conjugate partition of \(\mu_{k}\) writes
\[
\mu_{k}'=(r, \underbrace{k+1, \dotsc, k+1}_{\times q}).
\]
Accordingly, observe that, by definition of \(I_{\mu_{k}}\), the symmetric polynomials 
\[
e_{r}(Z_{1}, \dotsc, Z_{d-1})
\]
belong to \(I_{\mu_{k}}\) for \(r > k=(d-1)-(d-(k+1))\). This implies that the polynomial
\begin{equation}
\label{eq: sym1}
\prod\limits_{i=1}^{d-1}(1-tZ_{i}) \mod I_{\mu_{k}}
\end{equation}
has no term in \(t^{k+1}\). On the other hand, since \(I \subset I_{\mu_{k}}\), one can also expand
\begin{equation}
\label{eq: sym2}
\prod\limits_{i=1}^{d-1}(1-tZ_{i})
\equiv
\sum\limits_{i=0}^{d-1} t^{i} Z_{d}^{i}
\mod I_{\mu_{k}}.
\end{equation}
(Here, one uses that Newton polynomials span the ideal of symmetric polynomials).
The equalities \eqref{eq: sym1} and \eqref{eq: sym2} show that \(Z_{d}^{k+1}\) does indeed belong to \(I_{\mu_{k}}\), and thus one has the inclusion 
\[
I_{k} \subset I_{\mu_{k}}.
\]

In the other direction, prove by descending induction on the cardinal of \(S\) that the generators of \(I_{\mu_{k}}\) do belong to \(I_{k}\). When \(|S|\)=d, this is clear, as one recovers the elementary symmetric polynomials. Suppose that the conclusion holds for \(S\) with \((i+1) \leq |S| \leq d\). Take then \(S \subset \mathcal{Z}_{d}\), \(|S|=i\). If \(d_{i}(\mu)=0\), then there is nothing to prove. Otherwise, let \(i-d_{i}(\mu) < j \leq i\). Without loss of generality, suppose that \(d\) does not belong to \(S\), and consider
\(
S':=S \sqcup \{d\}.
\)
Note that, by induction, for any 
\[
\underbrace{(i+1)-d_{i+1}(\mu)}_{=i-d_{i}(\mu)-k} < j' \leq (i+1),
\]
the polynomial \(e_{j'}(S')\) belongs to \(I_{k}\). In particular, since one has the equality
\[
e_{j}(S')
=e_{j}(S) + Z_{d}e_{j-1}(S),
\]
one deduces that 
\[
e_{j}(S) \equiv Z_{d}e_{j-1}(S) \mod I_{k}.
\]
Repeating the argument until the index \(j-k-1\) (which is possible, as \(j-k > (i+1)-d_{i+1}(\mu)\)), one finds that
\[
e_{j}(S) 
\equiv 
Z_{d}e_{j-1}(S)
\equiv 
Z_{d}^2 e_{j-2}(S)
\equiv
\dotsb
\equiv Z_{d}^{k+1} e_{j-k-1}(S)
\mod I_{k}.
\]
But as \(Z_{d}^{k+1}\) belongs to \(I_{k}\), this shows that \(e_{j}(S)\) does indeed belong to \(I_{k}\). This finishes the proof.
\end{proof}
As an immediate corollary of the Propositions \ref{prop: reformulation} and  \ref{prop: DCP}, we therefore have:
\begin{corollary}
The invariant tensors \(E^{U}\) identify with the \(\mu_{k}\)-harmonic polynomials.
\end{corollary}

\subsection{A structural result concerning \(\mu\)-harmonic polynomials}
\label{sse: structural}

\subsubsection{General strategy to study \(\mu\)-harmonic polynomials}~
\label{ssse: strat}
Keeping the notations introduced in the previous Section \ref{sse: DCP}, we therefore wish to understand the \(\mu_{k}\)-harmonic polynomials (i.e. the vector space \(I_{k}^{\perp}\)), and in particular find a \textsl{simple} basis (a simple generating set is actually enough for our purposes). Here, as our ultimate goal is to prove a finite generation statement, the adjective \textsl{simple} means that the polynomials in the basis should be as reducible as possible: if the meaning of such a statement is not clear for now, it will become clear in the next Section \ref{se: fg}, where we will prove the finite generation of \((V^{\Diff})^{(k)}\).

The strategy to attack such a problem is rather standard, but technically involved (as the paper \cite{BG} may show). It is based on the following general result:
\begin{theorem}[\cite{Oberst96}]
\label{thm: PDE}
For any zero-dimensional ideal \(J \subset \C[Z_{1}, \dotsc, Z_{d}]\), one has the equality
\[
\dim_{\C} J^{\perp}
=
\dim_{\C} \C[Z_{1}, \dotsc, Z_{d}]/J.
\]
\end{theorem}
The strategy then goes as follows:
\begin{itemize}
\item{} exhibit a free family in \(J^{\perp}\) of cardinal \(M\);
\item{} exhibit a generating set of monomials in \(\C[Z_{1}, \dotsc, Z_{d}]/J\) of same cardinal \(M\).
\end{itemize}
By the above Theorem \ref{thm: PDE}, this implies that the free family in \(J^{\perp}\) is indeed a basis.

During our study of \(I_{k}^{\perp}\), the first part of the above strategy turned out to be quite manageable, as we knew what to expect. The surprisingly hard part was the second one:
we must admit that we spent quite some time trying to figure out a proof of the dimension of the quotient ring \(\C[Z_{1}, \dotsc, Z_{d}]/I_{k}\), which is
\[
\frac{d!}{(q!)^{k+1-r}((q+1)!)^{r}},
\]
without much success (at that time, we were obviously not aware of the paper \cite{BG}).
It turned out that, in our setting, the family of ideals considered was not big enough to proceed by induction (at least we did not find a way to do so). In the setting of DeConcini--Procesi ideals, there was enough room to do so, as explained in \cite{BG}[Section 4]. In the next Section \ref{ssse: structural}, we state a structural result from \cite{BG} that we will need.

\subsubsection{The structural result}
\label{ssse: structural}
In order to state it, let us first introduce some notations. If \(C=(\alpha_{1}, \dotsc, \alpha_{k})^{\intercal}\) is a column of integers, with \(\alpha_{i} \in \{1, \dotsc, d\}\), denote by \(\Delta(C)\) the following Vandermonde determinant:
\[
\Delta(C)
:=
\det 
\begin{pmatrix}
1 & Z_{\alpha_{1}} & \dotsc &Z_{\alpha_{1}}^{k-1}
\\
1 & Z_{\alpha_{2}} & \dotsc &Z_{\alpha_{2}}^{k-1}
\\
\cdot & \cdot & & \cdot
\\
\cdot & \cdot & & \cdot
\\
1 & Z_{\alpha_{k}} & \dotsc &Z_{\alpha_{k}}^{k-1}
\end{pmatrix}
\in \C[Z_{1}, \dotsc, Z_{d}].
\]
Now, for any standard tableau \(T \in \ST(\mu)\), define
\[
\Delta(T)
:=
\Delta(C_{1})
\dotsc 
\Delta (C_{d}),
\]
where the \(C_{i}\)'s are the columns of \(T\), and where, by convention, one sets \(\Delta(\emptyset)=1\).
As a corollary of \cite{BG}[Theorem 4.4], one deduces the following:
\begin{theorem}[\cite{BG}]
\label{thm: harmonic structural}
The \(\mu\)-harmonic polynomials are spanned, as a \(\C[\partial]\)-module, by
\[
\big(
\Delta(T)
\big)_{T \in \ST(\mu)}.
\]
\end{theorem}
\begin{remark}
Note that, by applying a monomial operator of the form
\[
\frac{\partial^{\alpha_{1}+\dotsc +\alpha_{d}}}{\partial Z_{1}^{\alpha_{1}} \dotsb \partial Z_{d}^{\alpha_{d}}}
\]
to \(\Delta(T)\), one keeps a polynomial that writes as a product (simply because the variables are separated). In particular, one indeed obtains a \textsl{simple} generating family of \(I_{\mu}^{\perp}\), in the sense of Section \ref{ssse: strat}.
\end{remark}

\begin{remark}
\label{rem: obs}
Note also that for \(k=(d-1)\), there is only one Vandermonde determinant \(W\), and one has:
\[
\C[\partial] \cdot W
=
\Vect
\{
W_{\alpha_{1}, \dotsc, \alpha_{d}}
\
|
\
0 \leq \alpha_{i} \leq (i-1)
\}.
\]
\end{remark}

Let us now apply this result in the case of \(I_{k}^{\perp}=I_{\mu_{k}}^{\perp}\simeq E^{U}=(F^{\otimes d})^{U_{k}}\), and let us suppose that \(d \geq (k+1)\).  In order to state the corollary in a convenient fashion, let us introduce some notations. First, let us explicitly distinguish the \(F\)'s appearing in the tensor product by writing
\[
E^{U}
=
(F^{1} \otimes \dotsb \otimes F^{d})^{U_{k}}.
\]
Now, for any permutation \(\sigma \in \mathcal{S}_{d}\), denote by
\[
E_{\sigma}:=(F^{\sigma(1)} \otimes \dotsb \otimes F^{\sigma(k+1)})^{U_{k}} 
\otimes
\dotsb 
\otimes 
(F^{\sigma((q-1)(k+1)+1)} \otimes \dotsb \otimes F^{\sigma(q(k+1))})^{U_{k}}
 \otimes 
 (F^{\sigma(q(k+1)+1)} \otimes \dotsb \otimes F^{\sigma(d)})^{U_{k}}.
\]
Observe that there is a natural map
\[
f_{\sigma}: E_{\sigma} \to E^{U},
\]
obtained simply by reordering the terms.
The previous Theorem \ref{thm: harmonic structural} now readily implies the following:
\begin{corollary}
\label{cor: structural}
The natural map induced by the \(f_{\sigma}\)'s
\[
 \bigoplus_{\sigma \in \mathcal{S}_{d}} E_{\sigma} \to E^{U}
\]
is surjective.
\end{corollary}
\begin{proof}
It suffices to observe (keeping in mind the above Remark \ref{rem: obs}) that Theorem \ref{thm: harmonic structural} implies in particular that the \(\mu_{k}\)-harmonic polynomials are spanned by the multi-linear polynomials
\begin{align*}
\Big(
\big\{W_{(\alpha_{1}, \dotsc, \alpha_{k+1})}(Y_{\sigma(1)}^{(\bigcdot)}, \dotsc, Y_{\sigma(k+1)}^{(\bigcdot)})
\
|
\
0 \leq \alpha_{i} \leq (i-1)
\big\}
\times 
\dotsb
\\
\times
\big\{W_{(\alpha_{1}, \dotsc, \alpha_{k+1})}(Y_{\sigma((q-1)(k+1)+1)}^{(\bigcdot)}, \dotsc, Y_{\sigma(q(k+1))}^{(\bigcdot)})
\
|
\
0 \leq \alpha_{i} \leq (i-1)
\big\}
\\
\times
\big\{W_{(\alpha_{1}, \dotsc, \alpha_{r})}(Y_{\sigma(q(k+1)+1)}^{(\bigcdot)}, \dotsc, Y_{\sigma(d)}^{(\bigcdot)})
\
|
\
0 \leq \alpha_{i} \leq (i-1)
\big\}
\Big)_{\sigma},
\end{align*}
where \(\sigma\) runs over the group of permutations \(\mathcal{S}_{d}\).

\end{proof}

\section{Finite generation of \((V^{\Diff})^{(k)}\), and generators}
\label{se: fg}

\subsection{Finite generation}
\label{sse: fg}
We now have everything we need to prove the main result of the paper: 

\begin{theorem}
\label{thm: finite generation}
The algebra \((V^{\Diff})^{(k)}\) is finitely generated. More precisely, if one fixes  for any \(1 \leq i \leq k+1\) a family 
\[
 \mathcal{B}_{i} \subset (V_{i}^{\Diff})^{(i-1)} 
 \]
  inducing a basis of \((V_{i}^{\Diff})^{(i-1)}/(V_{i}^{\Diff})^{(i-2)}\), then one has the equality:
  \[
  (V^{\Diff})^{(k)}
  =
  \C[\mathcal{B}_{1}, \mathcal{B}_{2}, \dotsc, \mathcal{B}_{k+1}].
  \]
\end{theorem}
\begin{proof}
Fix \(d \geq (k+1)\). By Lemma \ref{lemma: simple}, there exists a natural surjective map
\[
p: ((\underbrace{F\oplus \dotsb \oplus F}_{\times (N+1)})^{\otimes d})^{U} 
\twoheadrightarrow
S^{d}(\underbrace{F\oplus \dotsb \oplus F}_{\times (N+1)})^{U}
\simeq (V_{d}^{(k)})^{\Diff},
\]
where one recalls that \(F:=F_{k}\) and \(U:=U_{k}\) (see Section \ref{sse: harmonic} for notations).
The space on the left decomposes as a direct sum of summands of the form
\[
(F^{\otimes d})^{U}.
\]
By Corollary \ref{cor: structural}, for each of these summands, there exists a natural surjective map
\[
\pi:
\bigoplus_{\sigma \in \mathcal{S}_{d}} 
\big((F^{\otimes (k+1)})^{U_{k}}\big)^{\otimes q} \otimes (F^{\otimes r})^{U_{k}}
\twoheadrightarrow
(F^{\otimes d})^{U},
\]
where one recalls that \(d=q(k+1)+r\) is the euclidean division of \(d\) by \((k+1)\). 
Note that \((F^{\otimes r})^{U_{k}}\) actually identifies with \((F_{r-1}^{\otimes r})^{U_{r}}\) (recall that \(F_{r-1}\) is seen in a natural fashion as a sub-vector space of \(F=F_{k}\)).

Using the notations introduced during Section \ref{ssse: alternative}, consider, for any \(0 \leq i \leq k\), the basis of \((F_{i}^{\otimes (i+1)})^{U_{i}}\) given by
\[
(W_{(\alpha_{1}, \dotsc, \alpha_{i+1})})_{0 \leq \alpha_{j} \leq (j-1)}.
\]
Recall that the element \(W_{(\alpha_{1}, \dotsc, \alpha_{i+1})}\) identifies with a \textsl{multi-linear} polynomial of degree \((i+1)\) in the variables \((Y_{s}^{(t)})_{1 \leq s \leq (i+1); 0 \leq t \leq i}\).
Now, observe that the image of 
\[
W_{(\alpha_{1}^{1}, \dotsc, \alpha^{1}_{k+1})} \otimes \dotsb \otimes W_{(\alpha_{1}^{q}, \dotsc, \alpha^{q}_{k+1})} \otimes W_{(\alpha_{1}, \dotsc, \alpha_{r})}
\]
under the map \(p \circ \pi\) corresponds to the product polynomial 
\[
W_{(\alpha_{1}^{1}, \dotsc, \alpha^{1}_{k+1})}^{\subs}
\times
\dotsb
\times 
W_{(\alpha_{1}^{q}, \dotsc, \alpha^{q}_{k+1})}^{\subs}
\times 
W_{(\alpha_{1}, \dotsc, \alpha_{r})}^{\subs}
\in (V_{d}^{\Diff})^{(k)},
\]
where, in each term of the product, one has performed a certain substitution of variables of the form 
\[
\big(
(Y_{i}^{(\ell)})_{0 \leq \ell \leq k}
\mapsto
(X_{f(i)}^{(\ell)})_{0 \leq \ell \leq k}
\big)_{1 \leq i \leq d},
\]
where \(f\colon \{1, \dotsc d\}  \to \{0, \dotsc, N\}\) (the form of the substitution depends on the permutation \(\sigma \in \mathcal{S}_{d}\)).

Observe that the above implies that the natural graded map of algebra
\begin{equation}
\label{eq: surj gen}
S^{\bigcdot}(
V_{1}^{\Diff} \oplus \dotsb \oplus V_{k+1}^{\Diff}
)
\to
(V^{(k)})^{\Diff}
\end{equation}
is surjective: this establishes the sought finite generation. The final assertion in the statement follows from a simple induction on \(k \geq 0\). For \(k=0\), the statement is straightforward. Suppose that it holds for \((k-1) \geq 0\). Decomposing \(V_{k+1}^{\Diff}=\mathcal{B}_{k+1} \oplus S\), where \(S \subset (V_{k+1}^{(k-1)})^{\Diff}\), the above surjective map \eqref{eq: surj gen} shows that
\[
(V^{(k)})^{\Diff}
=
(V^{(k-1)})^{\Diff}[\mathcal{B}_{k+1}].
\]
The result now follows from the induction hypothesis. This finishes the proof.

\end{proof}

\subsection{Study of the generators}~
\label{sse: gen}
In view of the previous Theorem \ref{thm: finite generation}, we have to understand the space
\[
V_{d}^{\Diff}/(V_{d}^{\Diff})^{(d-2)}
=
(V_{d}^{\Diff})^{(d-1)}/(V_{d}^{\Diff})^{(d-2)}
\]
for any \(d \in \N\). Recall (see Section \ref{ssse: alternative}) that the space \(V_{d}^{\Diff}\) is obtained from the space \((F_{d-1}^{\otimes d})^{U_{d-1}}\) as follows:
\begin{itemize}
\item{} consider the basis of \((F_{d-1}^{\otimes d})^{U_{d-1}}\) (of cardinal \(d!\))
\[
(W_{(\alpha_{1}, \dotsc, \alpha_{d})})_{0 \leq \alpha_{j} \leq (j-1)}
\]
introduced in Section \ref{ssse: alternative}, where we recall that each element is identified with a multi-linear polynomial in the variables \((Y_{s}^{(t)})_{1 \leq s \leq d; 0 \leq t \leq (d-1)}\);
\item{}
perform substitutions of the form
\[
(Y_{s}^{(t)})_{1 \leq s \leq d; 0 \leq t \leq (d-1)}
\longmapsto
(X_{f(s)}^{(t)})_{1 \leq s \leq d; 0 \leq t \leq (d-1)},
\]
where \(f: \{1, \dotsc, d\} \to \{0, \dotsc, N\}\).
\end{itemize}
As the order of truncation is obviously preserved by these substitutions, we are first lead to understand the quotient 
\[
(F_{d-1}^{\otimes d})^{U_{d-1}}/(F_{d-2}^{\otimes d})^{U_{d-2}}.
\]
This is the object of the first Section \ref{ssse: gen model}. The study of the generators of \((V^{(k)})^{\Diff}\) is then carried over in the second Section \ref{ssse: gen general}

\subsubsection{The model case}~
\label{ssse: gen model}
In what follows, we are going to extract from the family of multi-linear polynomials
\[
(W_{(\alpha_{1}, \dotsc, \alpha_{d})})_{0 \leq \alpha_{j} \leq (j-1)}
\]
a family inducing a basis of the quotient \((F_{d-1}^{\otimes d})^{U_{d-1}}/(F_{d-2}^{\otimes d})^{U_{d-2}}\).
Let us thus consider the following set of multi-indexes
\[
\Sigma_{d}
:=
\Set{
\mathbi{\alpha} \in \N^{d}
\
|
\
\exists \ 1 \leq i \leq d, \alpha_{i}=0
\
\&
\
\text{
the \((d-1)\)-uple \(\overset{\wedge_{i}}{\mathbi{\alpha}}\) satisfies: \( \ \forall 1 \leq j \leq d-1\), \(\overset{\wedge_{i}}{\alpha_{j}} \leq j-1\)
}
}.
\]
One can partition the set \(\Sigma_{d}\) by considering the last index \(i\) such that \(\alpha_{i}=0\), i.e. by considering for \(1 \leq k \leq d\):
\[
\Sigma_{d}^{k}
:=
\Set{
\mathbi{\alpha} \in \Sigma_{d}
\
|
\
\alpha_{k}=0 
\
\&
\
\forall \ i > k, \alpha_{i} > 0
}.
\]
\begin{remark}~
Note that \(\Sigma_{d}^{1}\) is empty.
Note also that, if \(\mathbi{\alpha} \in \Sigma_{d}^{k}\), then \(\mathbi{\alpha}\) satisfies
\[
\overset{\wedge_{k}}{\alpha_{j}} \leq j-1
\]
for any \(1 \leq j \leq d-1\), so that the condition to be in \(\Sigma_{d}\) is satisfied for the index \(k\).
\end{remark}
One then has the following elementary lemma:
\begin{lemma}
\label{lemma: denombrement}
The set of indexes \(\Sigma_{d}\) has cardinal \(|\Sigma_{d}|=\frac{d!}{2}\).
\end{lemma}
\begin{proof}
Consider the previously introduced partition of \(\Sigma_{d}\)
\[
\Sigma_{d}
=
\Sigma_{d}^{d} \sqcup \Sigma_{d}^{d-1} \sqcup \dotsb \sqcup \Sigma_{d}^{1},
\]
and observe that the following equality holds for any \(1 \leq k \leq d\) (by convention, the empty product equals \(1\)):
\[
|\Sigma_{d}^{k}|
=
(k-1)! \times \prod\limits_{i=k-1}^{d-2} i.
\]
Indeed, to create an index in \(\Sigma_{d}^{k}\), one proceeds as follows:
\begin{itemize}
\item{} there are \(1\) choice for \(\alpha_{1}\), 2 choices for \(\alpha_{2}\),..., \((k-1)\) choices for \(\alpha_{k-1}\);
\item{} one has necessarily \(\alpha_{k}=0\);
\item{} there are \(k-1\) choices for \(\alpha_{k+1}\), \(k\) choices for \(\alpha_{k+2}\),..., \((d-2)\) choices for \(\alpha_{d}\) (because none of theses \(\alpha_{i}'s\) can be zero).
\end{itemize}
Observe that the previous equality can rewritten as:
\[
|\Sigma_{d}^{k}|
=
(d-1)! \times \frac{k-1}{d-1}.
\]
The result then follows from the identity:
\[
\sum\limits_{k=1}^{d} (d-1)! \frac{k-1}{d-1}
=
(d-1)!\frac{d(d-1)}{2(d-1)}
=
\frac{d!}{2}.
\]

\end{proof}

Let us then denote 
\[
\mathcal{F}_{d}
:=
(W_{\mathbi{\alpha}})_{\mathbi{\alpha} \in \Sigma_{d}}.
\]
This sub-family of the basis of \((F_{d-1}^{\otimes d})^{U_{d-1}}\) satisfies the property that we want:

\begin{proposition}
\label{prop: indep}
The family \(\mathcal{F}_{d}\) induces a basis of the quotient vector space
\[
(F_{d-1}^{\otimes d})^{U_{d-1}}/(F_{d-2}^{\otimes d})^{U_{d-2}}.
\]
\end{proposition}

\begin{proof}
Prove first that it is a free family.
Suppose by contradiction that one has the following non-trivial dependence relation modulo \((F_{d-2}^{\otimes d})^{U_{d-2}}\)
\begin{equation}
\label{eq: lin dep}
\sum\limits_{\mathbi{\alpha} \in \Sigma_{d}}
\lambda_{\mathbi{\alpha}} W_{\mathbi{\alpha}}
\equiv
0,
\end{equation}
where the coefficients \(\lambda_{\mathbi{\alpha}} \in \C\) are not all zero. Take the greatest \(1 \leq k \leq d\) such that there exists \(\mathbi{\alpha} \in \Sigma_{d}^{k}\) with \(\lambda_{\mathbi{\alpha}} \neq 0\). The key observations are now the following: 
\begin{itemize}
\item{} 
For any \(\mathbi{\alpha} \in \Sigma_{d}^{\ell}\), with \(\ell < k\), one has that \(W_{\mathbi{\alpha}}\) does not depend on the variable \(Y_{k}^{(d-1)}\).
Indeed, going back to the very matrix whose determinant gives \(W_{\mathbi{\alpha}}\), one immediately sees that the variable \(Y_{k}^{(d-1)}\) does not appear in the coefficients, simply because \(\alpha_{k}\) is greater or equal than \(1\).
\item{}
For any \(\mathbi{\alpha} \in \Sigma_{d}^{k}\), one has that the coefficient in front of \(Y_{k}^{(d-1)}\) equals (up to a sign) to:
\[
\coeff_{Y_{k}^{(d-1)}}(W_{\mathbi{\alpha}})
=
\pm
W_{\overset{\wedge_{k}}{\mathbi{\alpha}}}.
\]
Indeed, it suffices to develop the determinant defining \(W_{\mathbi{\alpha}}\) according to the \(k\)th column to deduce this equality.
In particular, one recognizes (up to a sign) distinct elements in the basis of \((F_{d-2}^{\otimes (d-1)})^{U_{d-2}}\).\footnote{There is a slight abuse of notation here. With our notations, the space \((F_{d-2}^{\otimes (d-1)})^{U_{d-2}}\) is (identified as) a subset of the space of multi-linear polynomial in the variables
\[
(Y_{1}^{(\ell)})_{\ell}, \dotsc, (Y_{d-1}^{(\ell)})_{\ell}.
\]
Here, these are rather seen as multi-linear polynomials in the variables
\[
(Y_{1}^{(\ell}))_{\ell}, \dotsc, (Y_{k-1}^{(\ell)})_{\ell}, (Y_{k+1}^{(\ell)})_{\ell}, \dotsc, (Y_{d}^{(\ell)})_{\ell}.
\]}
\end{itemize}
Now, going back to the equality \eqref{eq: lin dep}, and looking at the coefficient of \(Y_{k}^{(d-1)}\), one finds that
\[
\sum\limits_{\mathbi{\alpha} \in \Sigma_{d}^{k}}
\lambda_{\mathbi{\alpha}} 
(\pm W_{\overset{\wedge_{k}}{\mathbi{\alpha}}})
=0.
\]
But this contradicts the fact that, by the second item above, the family 
\[
\big(
W_{\overset{\wedge_{k}}{\mathbi{\alpha}}}
\big)_{\mathbi{\alpha} \in \Sigma_{d}^{k}}
\]
form a free family.

Prove now that it spans \((F_{d-1}^{\otimes d})^{U_{d-1}}/(F_{d-2}^{\otimes d})^{U_{d-2}}\).
Fix \(W \in (F_{d-1}^{\otimes d})^{U_{d-1}}\), and simplify it modulo the elements in \(\mathcal{F}_{d}\) as follows.
If \(\coeff_{Y_{d}^{(d-1)}}(W)=0\), move onto the next step. Otherwise, consider
\(
\coeff_{Y_{d}^{(d-1)}}(W),
\)
and note that this defines an element in \((F_{d-2}^{\otimes (d-1)})^{U_{d-2}}\). It is therefore a linear combination of elements of the form
\[
W_{\overset{\wedge_{d}}{\mathbi{\alpha}}},
\]
where \(\overset{\wedge_{d}}{\mathbi{\alpha}}\) satisfies \(\overset{\wedge_{d}}{\alpha_{i}} \leq (i-1)\) for any \(1 \leq i \leq d-1\). Now, note that for any \(\overset{\wedge_{d}}{\mathbi{\alpha}}\) as above, the element \(W_{(\overset{\wedge_{d}}{\mathbi{\alpha}},0)}\) 
 belongs to \(\mathcal{F}_{d}^{d}\)
 and has its coefficient in front of \(Y_{d}^{(d-1)}\) equal to \(W_{\overset{\wedge_{d}}{\mathbi{\alpha}}}\). Therefore, one deduces that one can simplify \(W\) modulo \(\Vect(\mathcal{F}_{d}^{d})\) so that
 \[
\coeff_{Y_{d}^{(d-1)}}(W)
 =
 0.
 \]
 Suppose therefore that it is the case, and discuss now according to the value of \(\coeff_{Y_{d-1}^{(d-1)}}(W)\). If it is zero, move onto the next step. Otherwise, as above, the coefficient \(\coeff_{Y_{d-1}^{(d-1)}}(W)\) defines an element in \((F_{d-2}^{\otimes (d-1)})^{U_{d-2}}\)\footnote{With the same slight abuse of notation as before.}, and can thus be written as a linear combination of elements of the form
\[
W_{
\overset{\wedge_{(d-1)}}{\mathbi{\alpha}}},
\]
where \(\overset{\wedge_{(d-1)}}{\mathbi{\alpha}}\) satisfies \(\overset{\wedge_{(d-1)}}{\alpha_{i}} \leq (i-1)\) for any \(1 \leq i \leq d-1\). Now, there is an important observation to be made: as \(W\) does not depend on the variable \(Y_{d}^{(d-1)}\), for any \(\overset{\wedge_{(d-1)}}{\mathbi{\alpha}}\) appearing in the linear combination, the coefficient \(\overset{\wedge_{(d-1)}}{\alpha_{d-1}}\) has to be greater or equal than \(1\). Accordingly, for any \(\overset{\wedge_{(d-1)}}{\mathbi{\alpha}}\) as above, denoting 
\[
\mathbi{\alpha}:=(\overset{\wedge_{(d-1)}}{\alpha_{1}}, \dotsc, \overset{\wedge_{(d-1)}}{\alpha_{d-2}},0, \overset{\wedge_{(d-1)}}{\alpha_{d-1}}),
\]
the element \(W_{\mathbi{\alpha}}\)
belongs to \(\mathcal{F}_{d}^{d-1}\).
Furthermore, it does not depend on the variable \(Y_{d}^{(d-1)}\) by construction and, up to a sign, the coefficient in front of \(Y_{d-1}^{(d-1)}\) is equal to \(W_{\overset{\wedge_{(d-1)}}{\mathbi{\alpha}}}\) . As above, one can thus simplify \(W\) modulo \(\Vect(\mathcal{F}_{d}^{d-1})\) so that
\begin{itemize}
\item{} the vanishing \(\coeff_{Y_{d}^{(d-1)}}(W)=0\) remains (as elements in \(\mathcal{F}_{d}^{d-1}\) do not depend on the variable \(Y_{d}^{(d-1)}\));
\item{} one has the vanishing \(\coeff_{Y_{d-1}^{(d-1)}}(W)=0\).
\end{itemize}

Continuing in this fashion, one indeed shows that, modulo
\[
\Vect(\mathcal{F}_{d})
=
\bigoplus_{i=1}^{d}
\Vect(\mathcal{F}_{d}^{d-i}),
\]
the element \(W\) does not depend on the variables \(Y_{1}^{(d-1)}, \dotsc, Y_{d}^{(d-1)}\), i.e. belongs to \((F_{d-2}^{\otimes d})^{U_{d-2}}\). This finishes the proof.

\end{proof}

\subsubsection{The general case}
\label{ssse: gen general}
By what we recalled in the beginning of the previous Section \ref{ssse: gen model}, we can obtain a generating set of \(V_{d}^{\Diff}/(V_{d}^{\Diff})^{(d-2)}\) from \(\mathcal{F}_{d}\) through various substitutions. Let us therefore define
\[
\mathcal{F}_{d}^{\subs}
:=
\big\{
\subs_{f}(W) 
\
|
\
W \in \mathcal{F}_{d},
f: \{1, \dotsc, d\} \to \{0, \dotsc, N\}
\text{
non-decreasing}
\big\},
\]
where the map \(\subs_{f}\) is simply the substitution map defined as follows
\[
(Y_{s}^{(t)})_{1 \leq s \leq d; 0 \leq t \leq (d-1)}
\longmapsto
(X_{f(s)}^{(t)})_{1 \leq s \leq d; 0 \leq t \leq (d-1)}.
\]
\begin{remark}
Observe that there many instances of \(W \in \mathcal{F}_{d}\) and non-decreasing
\(
f: \{1, \dotsc, d\} \to \{0, \dotsc, N\}
\)
yielding, by anti-linearity of the determinant, to the identical vanishing \(\subs_{f}(W)=0\). 
\end{remark}
Note that there is a natural bijection between the set of non-decreasing functions from \(\{1, \dotsc, d\}\) to \(\{0, \dotsc, N\}\) and the set of vectors in \(\N^{N+1}\) whose coordinates sum up to \(d\). Explicitly, the bijection is given by associating to any non-decreasing \(f: \{1, \dotsc, d\} \to \{0, \dotsc, N\}\) the vector \(\mathbi{m}_{f} \in \N^{N+1}\) defined as
\[
\mathbi{m}(f)=(|f^{-1}(0)|, \dotsc, |f^{-1}(N)|),
\]
where, by convention, the cardinal of the empty set is zero. We will rather adopt this point of view in what follows, and denote by \(f(\mathbi{m})\) the non-decreasing function associated to \(\mathbi{m} \in \N^{N+1}, |\mathbi{m}|=d\).

The goal is thus to extract a suitable free family of \(\mathcal{F}_{d}^{\subs}\), and show that it induces a basis of \(V_{d}^{\Diff}/(V_{d}^{\Diff})^{(d-2)}\). 
To this end, one proceeds as in the previous Section \ref{ssse: gen model}: the ideas are the same, it is just more tedious notation-wise.  
Consider the set \(\overline{\Sigma_{d}}\) made of sequences of uples
\(
\mathbi{\alpha}=(\mathbi{\alpha_{0}}, \dotsc, \mathbi{\alpha_{N}})
\)
such that
\begin{enumerate}
\item{} for any \(0 \leq i \leq N\),
\begin{itemize}
\item{} either \(\mathbi{\alpha_{i}}\) is empty;
\item{} either \(\mathbi{\alpha_{i}}\) is a sequence of natural numbers, i.e. 
\[
\mathbi{\alpha_{i}}
=
(\alpha_{i,1}, \dotsc, \alpha_{i,r_{i}}),
\]
with \(r_{i} \geq 1\);
\end{itemize}

\item{}
for any \(i\) such that \(\mathbi{\alpha_{i}}\) is not the empty set, one has the string of inequalities:
\[
0 \leq \alpha_{i,1} < \dotsb <  \alpha_{i,r_{i}} < r_{0}+ \dotsb + r_{i};
\]

\item{} one has the equality:
\[
r_{0} + \dotsc + r_{N}
=
d;
\]

\item{}
there exists \(0 \leq i \leq N\) such that
\begin{itemize}
\item{} \(\alpha_{i,1}=0\);

\item{} if \(r_{i}>1\), one has the strict inequality
\[
\alpha_{i,r_{i}} < r_{0}+\dotsb+r_{i}-1;
\]
\item{} for any \(j > i\), one has the strict inequality
\[
\alpha_{j,r_{j}} < r_{0}+\dotsb+r_{j}-1,
\]
with the convention that this holds if \(\mathbi{\alpha_{j}}\) is empty.
\end{itemize}

\end{enumerate}
To any such \(\mathbi{\alpha} \in \Sigma_{d}\), associate an element \(W_{\mathbi{\alpha}}^{\subs} \in \mathcal{F}_{d}^{\subs}\) as follows. Set
\[
\mathbi{m}(\mathbi{\alpha})
:=
(|\mathbi{\alpha}_{0}|, \dotsc, |\mathbi{\alpha}_{N}|) 
=
(r_{0}, \dotsc, r_{N}) 
\in \N^{N+1}
\]
(where, by convention, the length of the empty set is zero), and define
\[
W_{\mathbi{\alpha}}^{\subs}
:=
\subs_{f(\mathbi{m}(\mathbi{\alpha}))}(W_{\mathbi{\alpha}}).
\]
The rest of the Section \ref{ssse: gen general} is devoted to show that 
\[
\mathcal{G}_{d}
:=
\big(
W_{\mathbi{\alpha}}^{\subs}
\big)_{\mathbi{\alpha} \in \overline{\Sigma_{d}}}
\]
is the sought family.
\begin{remark}
The first three items \((1), (2)\) and \((3)\) are the conditions to ensure that \(W_{\mathbi{\alpha}}^{\subs}\) lies in what is called the \textsl{canonical basis} of \(V_{d}^{\Diff}\) in \cite{ESK}[Section 2.1](see Proposition 2.1.2 in \textsl{loc.cit}).
The last item \((4)\) is the analogue of the condition defining \(\Sigma_{d}\) in the model case (see Section \ref{ssse: gen model}).
\end{remark}

As in Section \ref{ssse: gen model}, partition the set \(\overline{\Sigma_{d}}\) by considering the last index \(0 \leq \ell \leq N\) such that \(\alpha_{\ell,1}=0\), i.e. by considering
\[
\overline{\Sigma_{d}^{\ell}}
:=
\{
\mathbi{\alpha} \in \overline{\Sigma_{d}}
\
|
\
\alpha_{\ell,1}=0
\
\&
\
\forall j > \ell, \alpha_{j,1} \geq 1
\},
\]
with the convention that the empty condition (i.e. for the indexes \(j > \ell\) such that \(\mathbi{\alpha}_{j}=\emptyset\)) always holds.
\begin{remark}
Note that, for \(\mathbi{\alpha} \in \overline{\Sigma_{d}^{\ell}}\), the index \(\ell\) is also the last index such that the condition \((4)\) holds.
\end{remark}

The following lemma counts the number of elements in \(\overline{\Sigma_{d}}\):
\begin{lemma}~

\label{lemma: denombre}
For \(d=1\),  the set \(\overline{\Sigma_{d}}\) has cardinal \(|\overline{\Sigma_{1}}|=N+1\).

For \(d \geq 2\), the set \(\overline{\Sigma_{d}}\) has cardinal \(|\overline{\Sigma_{d}}|=\frac{N(N+1)}{2} (N+1)^{d-2}\).
\end{lemma}
\begin{proof}
The case where \(d=1\) is clear, and corresponds to the fact that \(\mathcal{F}^{\subs}_{1}=\{X_{0}, \dotsc, X_{N}\}\).
Suppose that \(d \geq 2\), and consider the partition
\[
\overline{\Sigma_{d}}
=
\bigsqcup_{\ell=0}^{N}
\overline{\Sigma_{d}^{\ell}}.
\]
For any \(0 \leq \ell \leq N\), one has the equality
\begin{equation}
\label{eq: dénombre}
|\overline{\Sigma_{d}^{\ell}}|
=
\sum\limits_{\substack{
\mathbi{m} \in \N^{N+1}
\\
|\mathbi{m}|=d
}}
\binom{|\mathbi{m}|_{0}}{m_{0}}
\times
\dotsb
\times 
\binom{|\mathbi{m}|_{\ell-1}}{m_{\ell-1}}
\times
\binom{|\mathbi{m}|_{\ell}-2}{m_{\ell}-1}
\times 
\binom{|\mathbi{m}|_{\ell+1}-2}{m_{\ell+1}}
\times 
\dotsb
\times
\binom{d-2}{m_{N}},
\end{equation}
where one sets \(|\mathbi{m}|_{j}:=m_{0}+\dotsb+m_{j}\) the sum of the \((j+1)\) first coordinates of \(\mathbi{m}\). 
Indeed, fixing \(\mathbi{m} \in \N^{N+1}\) such that \(|\mathbi{m}|=d\), one counts the number of elements \(\mathbi{\alpha} \in \overline{\Sigma_{d}^{\ell}}\) satisfying
\(
\mathbi{m}(\mathbi{\alpha})
=
\mathbi{m}
\):
\begin{itemize}
\item{}
for each \(1 \leq j \leq \ell-1\), 
\begin{itemize}
\item{} 
either \(m_{j}=0\), in which case \(\mathbi{\alpha_{j}}\) must be the empty set: there is \(1=\binom{|\mathbi{m}|_{j}}{0}\) choice;

\item{}
either \(m_{j}>0\), in which case \(\mathbi{\alpha_{j}}\) must be an increasing sequence satisfying
\[
0 \leq \alpha_{j,1} < \dotsb < \alpha_{j,m_{j}} < |\mathbi{m}|_{j},
\]
and there are \(\binom{|\mathbi{m}|_{j}}{m_{j}}\) such choices;
\end{itemize}

\item{}
for \(j=\ell\), the value \(\alpha_{\ell,1}\) is prescribed to be zero, and 
\begin{itemize}
\item{} either \(m_{\ell}=1\), so that there is only \(1=\binom{|\mathbi{m}|_{\ell}-2}{0}\) choice (with the convention that for any \(n \in \Z\), one sets \(\binom{n}{0}=1\));

\item{} either \(m_{\ell}\geq 2\), in which case \(\mathbi{\alpha_{\ell}}\) must be an increasing sequence satisfying
\[
1 \leq \alpha_{\ell,2} < \dotsb < \alpha_{\ell,m_{\ell}} < |\mathbi{m}|_{\ell}-1,
\]
and there are \(\binom{|\mathbi{m}|_{\ell}-2}{m_{\ell}-1}\) such choices;
\end{itemize}

\item{}
for \(j>\ell\),
\begin{itemize}
\item{} 
either \(m_{j}=0\), in which case \(\mathbi{\alpha_{j}}\) must be the empty set, so that there is \(1=\binom{|\mathbi{m}|_{j}}{0}\) choice;

\item{} 
either \(m_{j}>0\), in which case \(\mathbi{\alpha_{j}}\) must be an increasing sequence satisfying
\[
1 \leq \alpha_{j,1} < \dotsb < \alpha_{j,m_{j}} < |\mathbi{m}|_{j}-1,
\]
and there are \(\binom{|\mathbi{m}|_{j}-2}{m_{j}}\) such choices.
\end{itemize}
\end{itemize}
Summing over the uples \(\mathbi{m} \in \N^{N+1}\) such that \(|\mathbi{m}|=d\), one finds the equality \eqref{eq: dénombre}.

Now, a straightforward computation shows that each element appearing in the sum \eqref{eq: dénombre} writes
\begin{align*}
&
|\mathbi{m}|_{\ell-1}\times \frac{(d-2)!}{m_{0}!\times \dotsb \times m_{\ell-1}!\times (m_{\ell}-1)! \times m_{\ell+1}! \times \dotsb \times m_{N}!}
\\
& = 
\sum\limits_{j=0}^{\ell-1}
\binom{d-2}{m_{0}, \dotsc, m_{j-1}, m_{j}-1, m_{j+1}, \dotsc, m_{\ell-1}, m_{\ell}-1, m_{\ell+1}, \dotsc, m_{N}}
\\
& =
\sum\limits_{j=0}^{\ell-1}
\binom{d-2}{\mathbi{m}-(\mathbi{e}_{j}+\mathbi{e}_{\ell})},
\end{align*}
where one has used the usual notation for multinomial coefficients, and where \((\mathbi{e}_{i})_{0 \leq i \leq N}\) is the canonical basis of \(\Z^{N+1}\).
Compute then that:
\begin{align*}
|\overline{\Sigma_{d}}|
=
\sum\limits_{\ell=0}^{N} 
|\overline{\Sigma_{d}^{\ell}}|
&
=
\sum\limits_{\ell=0}^{N} 
\sum\limits_{\substack{
\mathbi{m} \in \N^{N+1}
\\
|\mathbi{m}|=d
}}
\sum\limits_{j=0}^{\ell-1}
\binom{d-2}{\mathbi{m}-(\mathbi{e}_{j}+\mathbi{e}_{\ell})}
\\
&=
\sum\limits_{0 \leq j<\ell \leq N}
\sum\limits_{\substack{
\mathbi{m} \in \N^{N+1}
\\
|\mathbi{m}|=d}}
\binom{d-2}{\mathbi{m}-(\mathbi{e}_{j}+\mathbi{e}_{\ell})}.
\end{align*}
Finally, observe that for any \(0 \leq j < \ell \leq N\), one has the equality
\begin{align*}
\sum\limits_{\substack{
\mathbi{m} \in \N^{N+1}
\\
|\mathbi{m}|=d}}
\binom{d-2}{\mathbi{m}-(\mathbi{e}_{j}+\mathbi{e}_{\ell})}
=
\sum\limits_{\substack{
\mathbi{m} \in \N^{N+1}
\\
|\mathbi{m}|=d-2}}
\binom{d-2}{\mathbi{m}}
=
(N+1)^{d-2},
\end{align*}
where the last equality follows from Newton multinomial formula. Combined with the equality
\[
\sum\limits_{0 \leq j < \ell \leq N} 1=\binom{N+1}{2}=\frac{N(N+1)}{2},
\]
one obtains the sought formula for \(|\overline{\Sigma_{d}}|\), which finishes the proof of the lemma.
\end{proof}

Similarly to Proposition \ref{prop: indep}, we now have:
\begin{proposition}
\label{prop: indep general}
The family \(\mathcal{G}_{d}\) induces a basis of \((V_{d}^{\Diff})/(V_{d}^{\Diff})^{(d-2)}\).
\end{proposition}
\begin{proof}
The proof is similar to the one carried over in Proposition \ref{prop: indep}, just more tedious notation-wise.
One first shows that the family \(\mathcal{G}_{d}\) induces a free family in the quotient vector space \((V_{d}^{\Diff})/(V_{d}^{\Diff})^{(d-2)}\), using the triangular-shaped system induced by the decomposition
\[
\mathcal{G}_{d}
=
\bigsqcup_{\ell=0}^{N} \mathcal{G}_{d}^{\ell},
\]
where 
\(
 \mathcal{G}_{d}^{\ell}
:=
\{
W_{\mathbi{\alpha}}^{\subs}
\
|
\
\mathbi{\alpha} \in \overline{\Sigma_{d}^{\ell}}
\}
\). 
Then, one proves that (the images of) \(\mathcal{G}_{d}\) span \(V_{d}^{\Diff}/(V_{d}^{\Diff})^{(d-2)}\), following the same line of reasoning as in Proposition \ref{prop: indep}.
Complete details are omitted, as they are similar to the proof of Proposition \ref{prop: indep}.
\end{proof}

\subsection{Final form of the Main Theorem}~
To sum up this whole Section \ref{se: fg}, we have thus proved the following:
\begin{theorem}
\label{thm: final}
For any \(k \in \N\), the algebra of differentially homogeneous polynomials in \((N+1)\) variables \((V^{\Diff})^{(k)}\) is finitely generated by 
\begin{itemize}
\item{} \((N+1)\) generators \(\mathcal{G}_{1}\) in degree \(1\) (and order \(0\));
\item{} \(\frac{N(N+1)}{2}\) generators \(\mathcal{G}_{2}\) in degree \(2\) (and order \(1\));
\item{} \(\frac{N(N+1)}{2}\times N\) generators \(\mathcal{G}_{3}\) in degree \(3\) (and order \(2\));
\item{} \(\cdot\)
\item{} \(\cdot\)
\item{} \(\frac{N(N+1)}{2}\times N^{k-1}\) generators \(\mathcal{G}_{k+1}\) in degree \((k+1)\) (and order \(k\)).
\end{itemize}
Furthermore, this is a minimal set of generators.
\end{theorem}
\begin{proof}
Everything has been justified, except perhaps the minimality. 
To prove this, it suffices to observe the following two facts:
\begin{itemize}
\item{} the generators in degree \(i\) are minimal with respect to the degree amongst the differentially homogeneous polynomials that are truly of order \((i-1)\) (i.e. not of order strictly less than \((i-1)\));
\item{} by Proposition \ref{prop: indep general}, there does not exist any linear relation between the generators in degree \(i\) allowing to decrease their order.
\end{itemize}
\end{proof}

\section{Application: Projective compactifications of jet spaces of projective spaces}
\label{se: proj}
Let us fix \(N \geq 1\) a natural number, and \(k \geq 0\) a truncation level. 
Recall the geometric interpretation of \((V^{\Diff})^{(k)}\): this is the algebra of (generalized) functions on the \(k\)-jets of germ of holomorphic curves traced on the projective space \(\P^{N}\), i.e. on \(J_{k}\P^{N}\). The main Theorem \ref{thm: final} provides a description of this algebra:
\[
(V^{\Diff})^{(k)}
=
\C[\mathcal{G}_{1}, \dotsc, \mathcal{G}_{k+1}],
\]
where \(\mathcal{G}_{1}=\{X_{0}, \dotsc, X_{N}\}\), and where \((\mathcal{G}_{i})_{2 \leq i \leq (k+1)}\) is an explicit set of generators of degree \(i\) and order \((i-1)\).
\begin{example}~

\begin{enumerate}
\item{}
For \(k=0\), \((V^{\Diff})^{(0)}=\C[X_{0}, \dotsc, X_{N}]\). This is indeed the algebra of functions on points (i.e. \(0\)-jets) on \(\P^{N}\).

\item{}
For \(k=1\), \((V^{\Diff})^{(1)} = \C[(X_{i})_{0 \leq i \leq N}, (\Wronsk(X_{i}, X_{j}))_{0 \leq i<j \leq N}]\). This is the algebra of functions on points and tangent vectors (i.e. \(1\)-jets) on \(\P^{N}\).

\end{enumerate}
\end{example}
The reader will probably have recognized that, in the above example, the elements
\[
\big(
\Wronsk(X_{i}, X_{j}):=X_{i}X_{j}'-X_{j}X_{i}'
\big)_{0 \leq i<j \leq N}
\]
corresponds to a basis of \(H^{0}(\P^{N},\Omega_{\P^{N}}(2))\). This is no coincidence: there is natural one-to-one correspondence between differentially homogeneous polynomials and global sections of the so-called \textsl{Green--Griffiths vector bundles}. We refer to \cite{ESK}[Section 3] for more details, where the correspondence is detailed.

However, we would like to emphasize that we do \textsl{not} want to interpret it in this way: our general philosophy is that functions on points should not be put apart from functions on higher jets, as it is usually customary. That being said, there is therefore a natural projective compactification of the \(k\)th jet bundle over \(\P^{N}\):
\begin{definition-proposition}[Projective compactification of \(J_{k}\P^{N}\)]~
\label{def-prop: compactification}
For any \(k \in \N\), denote by
\[
(\P^{N})^{(k)}
:=
\Proj (V^{\Diff})^{(k)}.
\]
This is a projective variety, that compactifies in a natural fashion the \(k\)th jet bundle \(J_{k}\P^{N}\).
\end{definition-proposition}
\begin{proof}
The graded ring \((V^{\Diff})^{(k)}\) is finitely generated, in degree less or equal than \((k+1)\): the scheme \((\P^{N})^{(k)}\) is therefore a projective variety. Consider the subvariety
\[
D_{\infty}:=\{X_{0}=\dotsb X_{N}=0\},
\]
and set \(V:=(\P^{N})^{(k)}\setminus D_{\infty}\). The claim is that one has a natural isomorphism of (smooth) quasi-projective varieties:
\[
(V, \O_{V})
\simeq
(J_{k}\P^{N}, \O_{J_{k}\P^{N}}).
\]
On the one hand, recall that, from the point of view of algebraic geometry, the scheme \(J_{k}\P^{N} \overset{\pi}{\longrightarrow} \P^{N}\) is obtained as the glueing of the affine schemes 
\[
\Spec(\C[ \overset{\wedge_{i}}{(x_{0}, \dotsc, x_{N})}, \dotsc, \overset{\wedge_{i}}{(x_{0}^{(k)}, \dotsc, x_{N}^{(k)})}]).
\]
On the other hand, observe that the quasi-projective variety \(V\) is covered by the affine open subsets \(V_{i}:=\{X_{i} \neq 0\}\), and note that one has the following equality of affine schemes:
\[
V_{i}
=
\Spec(\C[ \overset{\wedge_{i}}{(x_{0}, \dotsc, x_{N})}, \dotsc, \overset{\wedge_{i}}{(x_{0}^{(k)}, \dotsc, x_{N}^{(k)})} ]).
\]
One checks that, in both cases, the glueings are the same, so that \(V=(\P^{N})^{(k)}\setminus D_{\infty}\) does indeed identify with \(J_{k}\P^{N}\).
\end{proof}

As one may expect, the projective variety \((\P^{N})^{(k)}\) embeds inside the weighted projective space
\[
\P(\underbrace{1, \dotsc,1}_{\times (N+1)}, \underbrace{2, \dotsc,2}_{\times \frac{N(N+1)}{2}}, \dotsc, \underbrace{k+1, \dotsc, k+1}_{\times \frac{N(N+1)}{2}(N+1)^{k-1}})
\]
via the generators of \((V^{\Diff})^{(k)}\). Accordingly, understanding the relations between the generators is of particular interest in order to better apprehend the compactification \((\P^{N})^{(k)}\). As mentioned in the Introduction, this is the object of an ongoing work.
\begin{remark}
Note that we already know that the compactification is not too bad, as \((V^{\Diff})^{(k)}\) is factorial: see Corollary \ref{cor: factoriality}.
\end{remark}

\bibliographystyle{alpha}
\bibliography{diff-homogeneous}

 \end{document}